\documentclass[11pt]{amsart}
\usepackage{amsmath,amssymb,amsthm,amscd,verbatim}
\usepackage{graphicx}
\usepackage{epsfig}
\usepackage{hyperref}

\newtheorem{thm}{Theorem}
\newtheorem{lem}[thm]{Lemma}
\newtheorem{conj}[thm]{Conjecture}

\newtheorem{cor}[thm]{Corollary}
\newtheorem*{thm*}{Theorem}

\theoremstyle{definition}

\theoremstyle{remark}

 \newcommand{\calP}{\mathcal{P}}

\newcommand{\calC}{\mathcal{C}}

\newcommand{\RR}{\mathbb{R}}

\newcommand{\bo}{\ensuremath{\mathrm{box}}}
\newcommand{\fw}{\ensuremath{\mathrm{fw}}}

\renewcommand{\le}{\leqslant}
\renewcommand{\leq}{\leqslant}
\renewcommand{\ge}{\geqslant}
\renewcommand{\geq}{\geqslant}

\begin{document}
\title{Boxicity of graphs on surfaces}
\author{Louis Esperet} \address{Laboratoire G-SCOP (CNRS,
  Grenoble-INP), Grenoble, France}
\email{louis.esperet@g-scop.fr}

\author{Gwena\"el Joret} \address{D\'epartement d'Informatique,
  Universit\'e Libre de Bruxelles, Brussels, Belgium}
\email{gjoret@ulb.ac.be}

\thanks{ This work was supported in part by the Actions de Recherche
  Concert\'ees (ARC) fund of the Communaut\'e fran\c{c}aise de
  Belgique.  Louis Esperet is partially supported by ANR Project
  Heredia under Contract \textsc{anr-10-jcjc-heredia}.  Gwena\"el
  Joret is a Postdoctoral Researcher of the Fonds National de la
  Recherche Scientifique (F.R.S.--FNRS)}

\date{}
\sloppy

\begin{abstract}
  The boxicity of a graph $G=(V,E)$ is the least integer $k$ for
  which there exist $k$ interval graphs $G_i=(V,E_i)$, $1 \le i \le
  k$, such that $E=E_1 \cap \cdots \cap E_k$.  Scheinerman proved in
  1984 that outerplanar graphs have boxicity at most two and Thomassen
  proved in 1986 that planar graphs have boxicity at most three.  In
  this note we prove that the boxicity of toroidal graphs is at most
  7, and that the boxicity of graphs embeddable in a surface $\Sigma$ of 
  genus $g$ is at most $5g+3$. This result yields improved bounds
  on the dimension of the adjacency poset of graphs on surfaces.
\end{abstract}
\maketitle

\section{Introduction}

Given a collection ${\calC}$ of subsets of a set $\Omega$, the
\emph{intersection graph} of ${\calC}$ is defined as the
graph\footnote{Graphs in this paper are finite, simple, and
  undirected.}  with vertex set ${\calC}$, in which two elements of
${\calC}$ are adjacent if and only if their intersection is non empty.
A \emph{$d$-box} is the Cartesian product $[x_1,y_1] \times \cdots
\times [x_d,y_d]$ of $d$ closed intervals of the real line.  The
\emph{boxicity} $\bo(G)$ of a graph $G$ is the least integer $d
\geq 1$ such that $G$ is the intersection graph of a collection of
$d$-boxes\footnote{It is sometimes considered that complete graphs
  have boxicity 0, but we find this confusing and hence do not take this
  convention. However we made sure that all results from papers
  following this convention that are quoted in this article are used
  safely in our proofs.}. An \emph{interval graph} is a graph of
boxicity one.

The \emph{intersection} $G_1 \cap \cdots \cap G_k$ of $k$ graphs
$G_1,\ldots ,G_k$ defined on the same vertex set $V$ is the graph
$(V,E_1 \cap \cdots \cap E_k)$, where $E_{i}$ ($1 \le i \le k$)
denotes the edge set of $G_{i}$.  Observe that the boxicity of a graph
$G$ can equivalently be defined as the least integer $k$ such that $G$ is
the intersection of $k$ interval graphs.

The concept of boxicity was introduced in 1969 by Roberts
\cite{Rob69}. It is used as a measure of the complexity of ecological
\cite{Rob76} and social \cite{Fre83} networks, and has applications in
fleet maintenance \cite{OR81}. Graphs with boxicity one (that is,
interval graphs) can be recognized in linear time. On the other hand,
Kratochv\'il \cite{Kra94} proved that determining whether a graph has
boxicity at most two is NP-complete.

Scheinerman proved in 1984 that outerplanar graphs have boxicity at
most two \cite{Sch84} and Thomassen proved in 1986 that planar graphs
have boxicity at most three \cite{Tho86}.  Other results on the
boxicity of graphs can be found in~\cite{ABC10, CS07,CFS08} and the
references therein.

Related to boxicity is the notion of adjacency posets of graphs, which
was introduced by Felsner and Trotter~\cite{FT00}. The {\em adjacency
  poset} of a graph $G=(V,E)$ is the poset $\calP_{G}=(W, \leq)$ with
$W=V \cup V'$, where $V'$ is a disjoint copy of $V$, and such that $u
\leq v$ if and only if $u=v$, or $u\in V$ and $v\in V'$ and $u,v$
correspond to two distinct vertices of $G$ which are adjacent in $G$.
Let us recall that the {\em dimension $\dim(\calP)$} of a poset
$\calP$ is the minimum number of linear orders whose intersection is
exactly $\calP$.

Felsner, Li, and Trotter~\cite{FLT10} recently showed that $\dim(\calP_{G})
\leq 5$ for every outerplanar graph $G$, and that $\dim(\calP_{G})
\leq 8$ for every planar graph $G$.  They also proved that
$\dim(\calP_{G}) \leq \frac{3}{2}\chi_a(G)(\chi_a(G)-1)$ for every
graph $G$ with $\chi_a(G) \geq 2$, where $\chi_a(G)$ denotes the
\emph{acyclic chromatic number} of $G$ (the least integer $k$ so that
$G$ can be properly colored with $k$ colors, in such a way that every
two color classes induce a forest).  Using a result of Alon, Mohar,
and Sanders~\cite{AMS96}, this implies that the dimension of
$\calP_{G}$ is $O(g^{8/7})$ when $G$ is embeddable in a surface of
genus $g$.   At the end of their paper, the
authors of~\cite{FLT10} write that it is likely that the $O(g^{8/7})$
upper bound on $\dim(\calP_{G})$ could be improved to $O(g)$.

In this note, we first observe that the boxicity can also be bounded
from above by a function of the acyclic chromatic number, namely
$\bo(G) \le \chi_a(G)(\chi_a(G)-1)$ for every graph $G$ with
$\chi_a(G) \geq 2$.  Next, using a result of Adiga, Bhowmick, and
Chandran~\cite{ABC10}, we relate $\dim(\calP_{G})$ to $\bo(G)$ by
observing that $\dim(\calP_{G}) \leq 2\,\bo(G)+\chi(G)+4$ for every
graph $G$ (here $\chi(G)$ denotes the chromatic number of $G$).  Then
we prove that $\bo(G) \leq 5g+3$ for every graph embeddable in a
surface of genus $g$.  This implies that $\dim(\calP_{G})$ is bounded by a linear function of $g$, thus
confirming the suggestion of Felsner {\em et al.}~\cite{FLT10}
mentioned above.  We also consider more closely the case of toroidal
graphs and show that every such graph has boxicity at most $7$, while
there are toroidal graphs with boxicity $4$. We conclude the paper
with several remarks and open problems about the boxicity of graphs on
surfaces.

\section{Boxicity and acyclic coloring}\label{sec:acy}

It can be deduced from~\cite{ABC10}, or directly from~\cite{Spe72},
that the graph obtained from the complete graph $K_n$ by subdividing
each edge precisely once has boxicity $\Theta(\log \log n)$. This
graph is 2-degenerate, hence the boxicity of a graph cannot be bounded
from above by a function of its degeneracy\footnote{ A graph $G$ is
  {\em $k$-degenerate} if every subgraph of $G$ 
  has a vertex with degree at most $k$. The {\em degeneracy} of $G$ is
  the least integer $k$ such that $G$ is $k$-degenerate.  } or 
chromatic number alone.  However, the boxicity can be bounded by a
function of the {\em acyclic} chromatic number, as we now show.\\

For a graph $G$ and a subset $X$ of vertices of $G$, we let $G[X]$ denote
the subgraph of $G$ induced by $X$, and let $G\setminus X$ denote the
graph obtained from $G$ by removing all vertices in $X$.

Consider a graph $G$ and a subset $X$ of vertices of $G$ together with
an interval graph $I$ on the vertex set $X$, such that $I$ is a
supergraph of $G[X]$. Assume without loss of generality that $I$ maps
all the vertices of $X$ to subintervals of some interval $[l,r]$ of
$\RR$. We call \emph{a canonical extension of $I$ to $G$} the
interval graph $I'$ defined by mapping the vertices of $X$ to their
corresponding intervals in $I$, and all other vertices of $G$ to the
interval $[l,r]$. Observe that a canonical extension of $I$ to $G$
is a supergraph of $G$.

\begin{lem}\label{lem:acy}
$\bo(G) \le \chi_a(G)(\chi_a(G)-1)$ for every graph $G$ with
  $\chi_a(G) \geq 2$.
\end{lem}

\begin{proof}
  Consider an acyclic coloring $c$ of $G$ with $k\geq 2$ colors. For any two
  distinct colors $i<j$, we consider the graph $G_{i,j}$ obtained from
  $G$ by adding an edge between every pair of non-adjacent vertices
  $u,v$ such that at most one of $u,v$ is colored $i$ or $j$.

  We first show that $G=\bigcap_{i<j}G_{i,j}$. Since all $G_{i,j}$'s
  are supergraphs of $G$, we only need to show that for every pair
  $u,v$ of non-adjacent vertices in $G$, there exist $i<j$ so that $u$
  and $v$ are not adjacent in $G_{i,j}$. Exchanging $u$ and $v$ if
  necessary, we may assume that $c(u) \leq c(v)$.  If $c(u)< c(v)$
  then $G_{c(u),c(v)}$ does not contain the edge $uv$. If $c(u)=c(v)$
  then $G_{c(u),k}$ (if $c(u) < k$) or $G_{1, c(u)}$ (if $c(u) = k$)
  does not contain the edge $uv$.

  We now prove that for every $i<j$, $\bo(G_{i,j})\le 2$. This implies
  that $\bo(G) \le 2 {k \choose 2} =k(k-1)$. Observe that since $c$ is
  an acyclic coloring of $G$, the subgraph $H_{i,j}$ of $G$ induced by
  the vertices colored $i$ or $j$ is a forest, and thus has boxicity
  at most two (this follows from~\cite{Sch84} but can also be proven
  independently fairly easily). Let $I_{i,j}$ and $J_{i,j}$ be two
  interval graphs on the vertex set $V(H_{i,j})$ such that
  $H_{i,j}=I_{i,j} \cap J_{i,j}$. Then $G_{i,j}$ is precisely the
  intersection of canonical extensions of $I_{i,j}$ and $J_{i,j}$
  to $G$, and thus has boxicity at most two.
\end{proof}

It follows from~\cite{NO03} that if every minor of a graph $G$ has
average degree at most $k$, then $\chi_a(G)\le \tfrac34 k^2+O(k)$. It 
is known that graphs with no $K_t$-minor have average degree $O(t
\sqrt{ \log t})$~\cite{Kos84,Tho84}, so they have acyclic chromatic
number $O(t^2 \log t)$. Lemma~\ref{lem:acy} then implies that their
boxicity is $O(t^4 \log^2 t)$.

Alon, Mohar and Sanders~\cite{AMS96} proved that graphs embeddable in a
surface of genus $g$ have acyclic chromatic number
$O(g^{4/7})$. Using Lemma~\ref{lem:acy}, this implies that such graphs
have boxicity $O(g^{8/7})$. In the next section we show 
that the boxicity is bounded by a linear function of $g$.

\section{Graphs on surfaces}\label{sec:sur}

We will prove that the boxicity of a graph embeddable in a surface of 
genus $g$ is $O(g)$, by induction on $g$. Before we do so, we
need three simple lemmas that will be useful throughout the induction.

\begin{lem}\label{lem:sur1}
  Let $G=(V,E)$ be a graph and let $X \subseteq V$ be such that $G[X]$
  contains $k$ pairwise disjoint pairs of non-adjacent vertices. Then 
  $\bo(G) \le \bo(G \setminus X) + |X| -k$.
\end{lem}

\begin{proof}
  Let $v_{2i-1},v_{2i}$ ($1 \le i \le k$) be $k$ pairwise disjoint
  pairs of non-adjacent vertices in $G[X]$, and let
  $v_{2k+1},\ldots,v_\ell$ be the remaining vertices of $X$. Consider
  $t$ interval graphs $I_1,\ldots,I_t$ on the vertex set $V\setminus
  X$ such that $G\setminus X=\bigcap_{i=1}^t I_i$. We will prove that
  $\bo(G) \le t+ \ell -k$.

  For every pair $v_{2i-1},v_{2i}$ ($1 \le i \le k$), we consider the
  interval graph $J_i$ defined as follows: $v_{2i-1}$ is mapped to
  $\{0\}$; $v_{2i}$ is mapped to $\{2\}$; the common neighbors of
  $v_{2i-1}$ and $v_{2i}$ in $G$ are mapped to $[0,2]$; the remaining
  neighbors of $v_{2i-1}$ are mapped to $[0,1]$; the remaining
  neighbors of $v_{2i}$ are mapped to $[1,2]$; and the remaining
  vertices are mapped to $\{1\}$. The graph $J_{i}$ is clearly a
  supergraph of $G$, and every non-neighbor of $v_{2i-1}$ or $v_{2i}$
  in $G$ is a non-neighbor of $v_{2i-1}$ or $v_{2i}$ (respectively) in
  $J_i$.

  Next, for every $i \in \{2k+1, 2k+2, \dots,\ell\}$, we define an
  interval graph $J_i$ as follows: $v_i$ is mapped to $\{0\}$; its
  neighbors in $G$ are mapped to $[0,1]$, and the remaining vertices
  are mapped to $\{1\}$. This is a supergraph of $G$, and every
  non-edge incident to $v_i$ in $G$ is a non-edge in $J_i$.

Let $I_1',\ldots I_t'$ denote canonical extensions of
$I_1,\ldots,I_t$ to $G$. We claim that $G$ is precisely the
intersection of the $I_i'$'s ($1\le i \le t$), and the $J_i$'s ($i \in
\{1,\ldots,k\} \cup \{2k+1,\ldots, \ell\}$). These graphs are clearly
supergraphs of $G$. Moreover, every non-edge $e$ of $G$ is a non-edge
in one of these graphs, since $e$ is either a non-edge in $G \setminus
X$, or is incident to some vertex $v_i$ with $i\in \{1, \dots,
\ell\}$.
\end{proof}

In all  subsequent applications of Lemma~\ref{lem:sur1}, $X$ will
induce a cycle in $G$. In this case we obtain 
$\bo(G) \le \bo(G\setminus X) + 3$ if $|X|=3$, and
$\bo(G) \le \bo(G\setminus X) + \lceil |X|/2 \rceil$ if $|X| \geq 4$. 

\begin{lem}\label{lem:sur2}
  Let $G=(V,E)$ be a graph and let $V_1$, $V_2$, $X$ be a partition of
  $V$ such that no edge of $G$ has an endpoint in $V_1$ and the other
  in $V_2$. Let $G_1$ be a graph obtained from $G[V_1 \cup X]$ by
  adding a (possibly empty) set of edges between pairs of vertices
  from $X$. Then $\bo(G) \le \bo(G_1)+\bo(G[V_2 \cup X])+1$. In
  particular $\bo(G) \le \bo(G[V_1 \cup X])+\bo(G[V_2 \cup X])+1$.
\end{lem}

\begin{proof}
  Consider $k$ interval graphs $I_1,\ldots,I_k$ on the vertex set $V_1
  \cup X$ such that $G_1=\bigcap_{i=1}^k I_i$, and $\ell$
  interval graphs $J_1,\ldots,J_\ell$ on the vertex set $V_2 \cup X$
  such that $G[V_2 \cup X]=\bigcap_{i=1}^\ell J_i$. 

  Let $I_1',\ldots,I_k'$ be canonical extensions of
  $I_1,\ldots,I_k$ to $G$, and let $J_1',\ldots,J_\ell'$ be
  canonical extensions of $J_1,\ldots,J_\ell$ to $G$. Finally, let $K$
  be the interval graph defined by mapping all vertices of $V_1$
  to $\{0\}$, all vertices of $V_2$ to $\{1\}$, and all 
  vertices of $X$ to $[0,1]$.

  It is clear that all the $I_i$'s, $J_i$'s and $K$ are supergraphs of
  $G$, and that every non-edge of $G$ appears in one of these
  graphs. Hence, $\bo(G)\le k + \ell+1$.
\end{proof}

The next lemma is a variation of~\cite[Lemma 7]{ACM11}, where the same
idea is used to obtain a slightly different result.

\begin{lem}\label{lem:sur2bis}
  Let $G=(V,E)$ be a graph and let $K \subseteq V$ be a clique in $G$. Let
  $H$ be a graph obtained from $G$ by removing some set of edges having
  their two endpoints in $K$. Then $\bo(G) \le 2 \,\bo(H)$.
\end{lem}

\begin{proof}
Consider $k$ interval graphs $I_1\ldots,I_k$ on the vertex set $V$
such that $H=\bigcap_{i=1}^k I_i$. For each $1 \le i\le k$, we define
two interval graphs $I_i^l$ and $I_i^r$ in the following way. Assume
that $I_i$ maps every vertex $v \in V$ to
some interval $[l(v),r(v)]$ of $\RR$. Set $a=\min \{l(v), v\in V$\} and
$b=\max \{r(v), v\in V\}$. The interval graph $I_i^l$ is obtained by
mapping each vertex $v$ to $[a,r(v)]$ if $v\in K$ and to $[l(v),r(v)]$
otherwise. The interval graph $I_i^r$ is obtained by
mapping each vertex $v$ to $[l(v),b]$ if $v\in K$ and to $[l(v),r(v)]$
otherwise. It is readily seen that $G=\bigcap_{i=1}^k (I_i^l\cap
I_i^r)$, which shows that $\bo(G)\le 2k$.
\end{proof}

Combining Lemmas~\ref{lem:sur2} and~\ref{lem:sur2bis}, we
immediately obtain the following corollary.

\begin{cor}\label{cor:cor1}
  Let $G=(V,E)$ be a graph and let $V_1$, $V_2$, $X$ be a partition of
  $V$ such that no edge of $G$ has an endpoint in $V_1$ and the other
  in $V_2$. Let $G_1$ be a graph obtained from $G[V_1 \cup X]$ by
  adding and/or removing any set of edges between pairs of vertices
  of $X$. Then $\bo(G) \le 2\,\bo(G_1)+\bo(G[V_2 \cup X])+1$. 
\end{cor}

We now turn to the boxicity of graphs on surfaces.  In this paper, a
{\em surface} is a non-null compact connected 2-manifold without
boundary.  We refer the reader to the book by Mohar and
Thomassen~\cite{MoTh} for background on graphs on surfaces.

Consider a graph $G$ embedded in a surface $\Sigma$.  For simplicity,
we use $G$ both for the corresponding abstract graph and for the
subset of $\Sigma$ corresponding to the drawing of $G$.  A cycle $C$
of $G$ is said to be {\em noncontractible} if $C$ is noncontractible
(as a closed curve) in $\Sigma$. Also, $C$ is called {\em surface
  separating} if $C$ separates $\Sigma$ in two connected pieces.  The
{\em facewidth} $\fw(G)$ of $G$ is the least integer $k$ such that $\Sigma$
contains a noncontractible simple closed curve intersecting $G$ in
$k$ points.  If $G$ triangulates $\Sigma$ then its facewidth is equal
to the length of a shortest noncontractible cycle in $G$.  Two cycles
of $G$ are {\em (freely) homotopic} in $\Sigma$ if there is a
continuous deformation mapping one to the other.

The following well-known fact (often called the \emph{3-Path
  Property}) will be used: If $P_{1}, P_{2},P_{3}$ are three
internally disjoint paths with the same endpoints in an embedded
graph, and $P_{1}, P_{2}$ are such that $P_{1} \cup P_{2}$ is a
noncontractible cycle, then at least one of the two cycles $P_{1} \cup
P_{3}$, $P_{2} \cup P_{3}$ is also noncontractible (see for
instance~\cite[Proposition 4.3.1]{MoTh}). This implies the following
lemma.

\begin{lem}\label{lem:sur3}
  Suppose that $C$ is a noncontractible cycle of a graph $G$ embedded
  in a surface. Then there exists a noncontractible induced cycle
  $C'$ of $G$ with $V(C') \subseteq V(C)$.
\end{lem}

The next lemma is a standard fact about noncontractible cycles
in embedded graphs, see~\cite[Chapter 4.2]{MoTh}. 

\begin{lem}\label{lem:sur4}
  Suppose that $C$ is a noncontractible cycle of a graph $G$ embedded
  in a surface of genus $g \geq 1$. Then each
  component of $G \setminus V(C)$ is embeddable in a
  surface of genus $g-1$.
\end{lem}

Recall that Thomassen~\cite{Tho86} proved that $\bo(G)\le 3$ for every
planar graph $G$.  We are now ready to state and prove the main result
of this note, extending Thomassen's bound to general surfaces. 

\begin{thm}\label{th:sur}
  Let $G$ be a graph embedded in a surface $\Sigma$ of genus
  $g$. Then $\bo(G)\le 5g+3$. 
\end{thm}

\begin{proof}
  We prove the result by induction on $g$. If $g=0$ the bound follows
  from~\cite{Tho86}, so we can assume that $g\ge 1$. We can also
  assume that $G$ triangulates $\Sigma$, since $G$ is an induced
  subgraph of a triangulation\footnote{We can first assume that $G$ is connected and
    has no embedding on a surface with lower genus, and then
    using~\cite[Proposition 3.4.1]{MoTh} that $G$ has an embedding in
    $\Sigma$ in which every face is homeomorphic to an open disk. In
    this embedding, for every face $f$ consider a boundary walk $e_1,\ldots,e_k$ 
	around $f$, add a triangulated polygon on $k$ vertices $v_1,\ldots,v_k$ inside $f$, 
	and join each $v_i$ with the endpoints of $e_i$.} of $\Sigma$ and
  the boxicity is monotone by taking induced subgraphs.

  First suppose that $\fw(G) \le 5$. Since $G$ is a triangulation,
  there exists a noncontractible cycle $C$ of length at most 5. Using
  Lemma~\ref{lem:sur3}, we can further assume that $C$ is an induced
  cycle of $G$. The boxicity of a graph is clearly the maximum
  boxicity of its components. Thus by Lemma~\ref{lem:sur4} we obtain by
  induction that $\bo(G\setminus V(C))\le 5(g-1)+3$, 
  and by Lemma~\ref{lem:sur1} that $\bo(G)\le 5(g-1)+3 + 3 \leq 5g + 3$.

\smallskip

  From now on we assume that $\fw(G) \ge 6$, and we consider a
  shortest noncontractible cycle $C$ in $G$. It follows from
  Lemma~\ref{lem:sur3} that $C$ is an induced cycle (otherwise, we
  could shorten it). Let $V'$ be the set of vertices from
  $V(G)\setminus V(C)$ having at least one neighbor in $C$. We will
  construct a graph $H$ with $\bo(H)\le 2$ such that $H$ can be
  obtained from $G[V' \cup V(C)]$ by only adding and removing edges having
  both endpoints in $V'$. By
  Corollary~\ref{cor:cor1}, we then have $\bo(G) \le 2\,\bo(H)+\bo(G\setminus
  V(C))+1\le \bo(G\setminus V(C))+5$.  This in turn implies the
  theorem since Lemma~\ref{lem:sur4} and the induction
  hypothesis give that $\bo(G\setminus V(C))\le 5(g-1) + 3$, implying
  $\bo(G)\le 5g + 3$.

\medskip

  We remark that every vertex from $V'$ has at most three neighbors
  in $C$. More precisely, if some vertex of $V'$ does not belong to
  one of these four disjoint sets:
\begin{itemize}
\item$S_1$: the set of vertices of $V'$ with exactly one neighbor in $C$;
\item$S_2$: the set of vertices of $V'$ with exactly two neighbors in $C$
  and such that these vertices are consecutive in $C$;
\item$S_3$: the set of vertices of $V'$ with exactly two neighbors in $C$
  and such that these vertices are at distance two in $C$;
\item$S_4$: the set of vertices of $V'$ with exactly three neighbors in $C$
  and such that these vertices are consecutive in $C$;
\end{itemize}
then, since $C$ has length at least 6, the 3-Path Property implies that
$G$ contains a noncontractible cycle that is shorter than $C$, which is a
contradiction. 

\smallskip

Enumerate the vertices of $C$ as $v_1,\ldots,v_k$, in order. Let $H$
be the intersection graph of the 2-dimensional boxes depicted in
Figure~\ref{fig:H}. Vertices $v_1$, $v_{k-1}$, and $v_k$ are mapped to
boxes with corners $(-1,0)$ and $(0,3)$, $(k-3,0)$ and $(k-2,3)$,
$(0,3)$ and $(k-3,4)$, respectively. Every other vertex $v_i$ is
mapped to the box with corners $(i-2,i\mod 2)$ and $(i-1,1+i\mod
2)$. Consider now the vertices of $V'$: 
Each vertex $v\in S_1$ is mapped to an arbitrary point in the interior of the box 
of $v_i$, where $v_{i}$ is the unique neighbor of $v$ in $C$. 
Each vertex $v\in S_2$ is mapped to the point 
$(0, 3)$ if $v$ is adjacent to $v_{1}$ and $v_{k}$ in $C$;
to the point $(k-3, 3)$ if $v$ is adjacent to $v_{k-1}$ and $v_{k}$ in $C$, and 
to the point $(i-1, 1)$ if $v$ is adjacent to $v_{i}$ and $v_{i+1}$ in $C$ for some 
$i \in \{1, \dots, k-2\}$. 

\smallskip

Consider now a vertex $v\in S_3$ with neighbors $v_i$ and $v_{i+2}$ in
$C$, where $i\in \{1, \dots, k-3\}$. Then $v$ is mapped to the
horizontal segment with endpoints $(i-1,\tfrac12 +i\mod 2)$ and
$(i,\tfrac12 +i\mod 2)$. Vertices of $S_3$ that are adjacent to $v_2$
and $v_k$ in $C$ are mapped to the vertical segment with endpoints
$(\tfrac12,1)$ and $(\tfrac12,3)$.  Vertices of $S_3$ that are
adjacent to $v_{k-2}$ and $v_k$ in $C$ are similarly mapped to the
vertical segment with endpoints $(k-\tfrac72,1 + k\mod 2)$ and
$(k-\tfrac72,3)$.  Vertices of $S_{3}$ that are adjacent to $v_1$ and
$v_{k-1}$ in $C$ are mapped to the horizontal segment with endpoints
$(0,\tfrac52)$ and $(k-3,\tfrac52)$.

\smallskip

Each vertex $v\in S_4$ that is adjacent to $v_i,v_{i+1},v_{i+2}$ in $C$ 
where $i\in \{1, \dots, k-3\}$ is mapped to the horizontal segment with
endpoints $(i-1,1)$ and $(i,1)$. 
Each vertex $v\in S_4$ that is adjacent to $v_1,v_{2},v_{k}$ in $C$ 
is mapped to the vertical segment with endpoints $(0,1)$ and $(0,3)$. 
Each vertex $v\in S_4$ that is adjacent to $v_{k-2},v_{k-1},v_{k}$ in $C$ 
is mapped to the vertical segment with endpoints $(k-3,1)$ and $(k-3,3)$. 
Finally, each vertex $v\in S_4$ that is adjacent to $v_{1},v_{k-1},v_{k}$ in $C$ 
is mapped to the horizontal segment with endpoints $(0,3)$ and $(k-3,3)$.

\begin{figure}[htbp]
\centering
\includegraphics[scale=0.8]{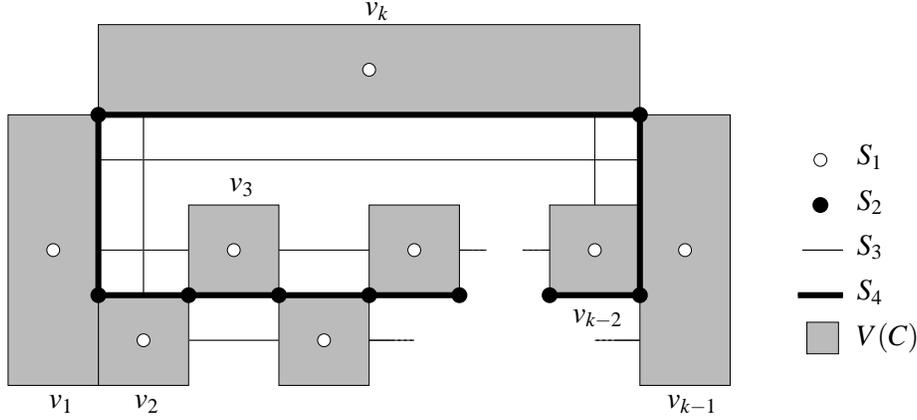}
\caption{The graph $H$ as intersection of 2-dimensional boxes.} \label{fig:H}
\end{figure}

It follows from the definition of $H$ that the vertices of $V(C)$
induce the cycle $C$ in $H$ and that the neighbors on $V(C)$ of any vertex
$v \in V'$ are the same in $G$ and in $H$. This shows that $H$ can be
obtained from $G[V' \cup V(C)]$ by only adding and removing edges
having both endpoints in $V'$, which concludes the proof.
\end{proof}

Theorem~\ref{th:sur} implies that toroidal graphs have boxicity at
most 8. We improve on this bound by using the following remarkable
result of Schrijver~\cite{S93}: Every graph embedded in the torus with
facewidth $k$ contains $\lfloor 3k/4 \rfloor$ vertex-disjoint
noncontractible cycles.  (Note that on the torus, such cycles are
necessarily homotopic.)

\begin{thm}\label{th:torus}
  $\bo(G)\le 7$ for every toroidal graph $G$.
\end{thm}

\begin{proof}
  Again, we may assume that $G$ triangulates the torus.

Assume first that $\fw(G) \le 5$. Since $G$ is a triangulation, there
exists a noncontractible cycle $C$ of length at most 5 such that
$G \setminus V(C)$ is planar. Using Lemma~\ref{lem:sur3}, we can
further assume that $C$ is an induced cycle of $G$. Then, using
Lemma~\ref{lem:sur1} and the result of Thomassen about the
boxicity of planar graphs, we deduce that $\bo(G)\le
3+3=6$. 

Assume now that $\fw(G) \ge 6$.  The aforementioned result of
Schrijver implies that $G$ contains 4 pairwise vertex-disjoint
noncontractible cycles, say $C_1,C_2,C_3,C_4$ in this order. Because
of $C_2$ and $C_4$ there are no edges between $C_1$ and $C_3$ in
$G$. We may further assume by Lemma~\ref{lem:sur3} that $C_1$ and
$C_3$ are induced cycles in $G$.  (Observe that every noncontractible
cycle in $G[V(C_{1})]$ or in $G[V(C_{3})]$ is again homotopic to the
four cycles $C_1,C_2,C_3,C_4$, because such a cycle is vertex-disjoint
from $C_{2}$ and $C_{4}$.)  The removal of $C_1$ and $C_3$ cuts the
torus into two connected pieces $\Sigma_1$ and $\Sigma_2$. Let $V_i$
($i=1,2$) be the set of vertices lying on $\Sigma_i$, and set
$X=V(C_1) \cup V(C_3)$. Since $G[V_1 \cup X]$ and $G[V_2 \cup X]$ are
planar, it follows from Lemma~\ref{lem:sur2} that $\bo(G)\le 3+3+1=7$.
\end{proof}

It was proved in~\cite{Rob69} that for every $n\ge 1$, the graph $G_{2n}$
obtained from $K_{2n}$ by removing a perfect matching has boxicity
exactly $n$. Since $G_8$ can be embedded on the torus (see
Figure~\ref{fig:k8}), there exist toroidal graphs with boxicity four.

\begin{figure}[htbp]
\centering
\includegraphics[scale=0.6]{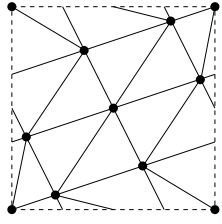}
\caption{A toroidal embedding of the graph obtained from $K_8$ by
  removing a perfect matching (the four corners correspond to the same
  vertex).} \label{fig:k8}
\end{figure}

Recall that, for a graph $G=(V,E)$, the adjacency poset $\calP_{G}$ of $G$ 
is defined as the poset $\calP_{G}=(W, \leq)$ with $W=V \cup V'$, where
$V'$ is a disjoint copy of $V$, and $u \leq v$ if and only if $u=v$,
or $u\in V$ and $v\in V'$ and $u,v$ correspond to two distinct
vertices of $G$ which are adjacent in $G$.  Let $\calP^{*}_{G}$ denote
the poset obtained from $\calP_{G}$ by adding that $u \leq v$ for
every $(u,v)\in V \times V'$ such that $u$ and $v$ correspond to the
same vertex of $G$.  Adiga, Bhowmick, and Chandran~\cite{ABC10}
recently proved that $\dim(\calP^{*}_{G})/2 - 2 \leq \bo(G) \leq
2\dim(\calP^{*}_{G})$ for every graph $G$.  Using this result, we may
bound the dimension of $\calP_{G}$ as follows.

\begin{thm}
\label{th:dim_box} 
$\dim(\calP_{G})\le 2\, \bo(G)+\chi(G)+4$ for every graph $G=(V,E)$. 
\end{thm}

\begin{proof}
  We have that $\dim(\calP^{*}_{G})\le 2\, \bo(G)+4$ by the
  aforementioned result of Adiga {\em et al.}~\cite{ABC10}, thus it is
  enough to show that $\dim(\calP_{G}) \leq \dim(\calP^{*}_{G}) +
  \chi(G)$. Consider a (proper) coloring $V_{1}, V_{2}, \dots, V_{k}$
  of $G$ with $k=\chi(G)$ colors, and let $V'_{1}, V'_{2}, \dots,
  V'_{k}$ denote the corresponding partition of $V'$.  For
  $i\in\{1,\dots, k\}$, let $\mathcal{L}_{i}=(W, \leq_{i})$ be an
  arbitrary linear order satisfying that $$V_{1} \cup \cdots \cup
  V_{i-1} \cup V_{i+1} \cup \cdots \cup V_{k} \leq_{i} V'_{i} \leq_{i}
  V_{i} \leq_{i} V'_{1} \cup \cdots \cup V'_{i-1} \cup V'_{i+1} \cup
  \cdots \cup V'_{k}.$$  (Here $A \leq_{i} B$ means that $u \leq_{i} v$
  for every $u\in A$ and $v \in B$.)  Then it is easily
  checked that each $\mathcal{L}_{i}$ is a linear extension of
  $\calP_{G}$, and that the intersection of these $k$ linear orders
  with $\calP^{*}_{G}$ is exactly $\calP_{G}$. It follows that
  $\dim(\calP_{G}) \leq \dim(\calP^{*}_{G}) + k$, as desired.
\end{proof}

\begin{cor}
  Let $G$ be a graph embeddable in a surface $\Sigma$ of genus
  $g \geq 1$. Then $\dim(\calP_{G}) \leq 10g + \tfrac12\,(27 + \sqrt{1 + 48g})$ if $\Sigma$ is
  orientable,  and $\dim(\calP_{G}) \leq 10g + \tfrac12\,(27 + \sqrt{1 + 24g})$ otherwise. 
\end{cor}

\begin{proof}
This follows from Theorems~\ref{th:sur}  
and~\ref{th:dim_box}, and Heawood's upper
bound on the chromatic number of $G$, namely  
$\chi(G) \leq  \tfrac12\,(7 + \sqrt{1 + 48g})$ if $\Sigma$ is orientable, and 
$\chi(G) \leq  \tfrac12\,(7 + \sqrt{1 + 24g})$ otherwise. 
\end{proof}

This confirms what Felsner, Li, and Trotter~\cite{FLT10} suggested as
an improvement of their result.

\section{Open questions}\label{sec:que}

The first question is whether the bounds obtained in
Section~\ref{sec:sur} are best possible.  We believe that the boxicity
of graphs embeddable in a surface of genus $g$ should rather be
$O(\sqrt{g})$. Since the complete graph $K_{2n}$ with a perfect
matching removed has boxicity $n$, this would be optimal. This example
also shows that the boxicity of graphs with no $K_t$-minor can be 
linear in $t$, while we only know a $O(t^4 \log^2 t)$  
upper bound (see the remark after Lemma~\ref{lem:acy}).\\

The \emph{edgewidth} of a graph $G$ embedded in a surface $\Sigma$ is
the length (number of edges) of a shortest noncontractible cycle in
$G$. Kawarabayashi and Mohar~\cite{KM10} proved that for every fixed
surface $\Sigma$, graphs embeddable in $\Sigma$ with sufficiently
large edgewidth are acyclically 7-colorable. It then follows from
Lemma~\ref{lem:acy} that these graphs have boxicity at most 42. We
believe that the following stronger statement is true:

\begin{conj}
For every fixed surface $\Sigma$ there exists an integer $e_\Sigma$ so that
every graph $G$ embeddable on $\Sigma$ with edgewidth at least $e_\Sigma$ has
boxicity at most three.
\end{conj}

It follows from a theorem of Thomassen~\cite{Tho86} that triangle-free
planar graphs have boxicity at most two. Since there exist trees that are not
interval graphs, a natural question is whether, for every surface
$\Sigma$, graphs embeddable in $\Sigma$ and having sufficiently large
girth (length of a shortest cycle) have boxicity at most two. We prove
that the following slightly weaker statement holds:

\begin{thm}
  For every fixed surface $\Sigma$ there exists some integer
  $g_\Sigma$ such that every graph with girth at least $g_\Sigma$
  embeddable in $\Sigma$ has boxicity at most 4.
\end{thm}

\begin{proof}
  It is well-known (see~\cite{ACKKR04}) that there exists an integer
  $g_\Sigma$ such that the vertex set of every graph $G$ embeddable on $\Sigma$
  and having girth at least $g_\Sigma$ can be partitioned into a forest $F$
  and a stable set $S$, in such way that every two vertices of $S$ are
  at distance at least three in $G$.

  Consider the graph $G_1$ obtained from $G$ by adding an edge
  between every pair of non-adjacent vertices $u,v$, such that at
  least one of $u,v$ is in $S$. As remarked in the proof of
  Lemma~\ref{lem:acy}, $\bo(G_1)\le 2$. Observe now that every vertex
  of $F$ has at most one neighbor in the  stable set $S$. Using
  this property, it can be deduced from~\cite[Proof of Theorem
  1]{CFS08} that the graph $G_2$ obtained from $G$ by adding all
  possible edges between pairs of vertices of $F$ has boxicity at most
  two. Since $G=G_1 \cap G_2$, it follows that $\bo(G)\le 4$.
\end{proof}

\paragraph{\bf Acknowledgment}

We would like to thank Bojan Mohar for interesting discussions about
surface separating noncontractible cycles, and the three anonymous referees for
their helpful comments on an earlier version of this paper. We are especially 
grateful to a referee whose suggestions lead to an improved constant in 
Theorem~\ref{th:sur}.

\end{document}